\documentclass[11pt,letterpaper,reqno]{amsart}
\usepackage{tikz}
\usetikzlibrary{positioning, shapes.geometric, arrows.meta, calc}
\usepackage{amssymb}
\usepackage{amsmath}
\usepackage{amsthm}
\usepackage{amsfonts}
\usepackage{bbm}
\usepackage{enumitem}
\usepackage{pgfplots}
\pgfplotsset{compat=1.18}
\usepackage{booktabs}
\usepackage{graphicx}
\usepackage[T1]{fontenc}
\usepackage{doi}
\usepackage{float}
\usepackage{bookmark}
\usepackage{hyperref}
\hypersetup{pdfstartview={FitH}}

\newtheorem{theorem}{Theorem}[section]
\newtheorem{lemma}[theorem]{Lemma}

\newtheorem{problem}[theorem]{Problem}
\newtheorem{conjecture}[theorem]{Conjecture}
\theoremstyle{definition}

\newtheorem{remark}[theorem]{Remark}

\numberwithin{equation}{section}

\newcommand{\C}{\mathbb{C}}

\newcommand{\D}{\mathbb{D}}
\newcommand{\whC}{\widehat{\mathbb{C}}}

\newcommand{\capf}{\operatorname{cap}}
\newcommand{\DR}{\mathbb{D}_R}

\begin{document}

\title{On Fuchs's additive intersection problem for the hyperbolic metric}

\author[Y.~He]{Yixin He}
\address{School of Mathematical Sciences, Fudan University, Shanghai 200433, P. R. China}
\email{hyx717math@163.com}

\author[Q.~Tang]{Quanyu Tang}
\address{School of Mathematics and Statistics, Xi'an Jiaotong University, Xi'an 710049, P. R. China}
\email{tang\_quanyu@163.com, tangquanyu@stu.xjtu.edu.cn}



\begin{abstract}
For hyperbolic domains $D_1,D_2\subset \{z\in\mathbb C:|z|<R\}$ and $z\in D_1\cap D_2$, we consider the ratio
$$
\frac{\lambda_{D_1\cap D_2}(z)}
{\lambda_{D_1}(z)+\lambda_{D_2}(z)}.
$$
We solve a problem of W.\,H.\,J.~Fuchs by proving that the supremum of this ratio is $+\infty$ when $D_1$ and $D_2$ range over all hyperbolic domains. If $D_1$ and $D_2$ are further assumed to be simply connected, then the supremum is $1$. We also show that the infimum of this ratio is $\frac12$ in both settings, and that the value $\frac12$ is attained if and only if $D_1=D_2$.
\end{abstract}

\maketitle

\section{Introduction}\label{sec:introduction}

The hyperbolic metric is one of the basic conformal invariants of a plane domain. It encodes the intrinsic geometry of the domain and plays a central role in geometric function theory; see, for instance, \cite{Ahlfors1973,BeardonMinda2007}. 

Recall that a domain \(X\) in the Riemann sphere \(\widehat{\C}\) is called
\emph{hyperbolic} if it carries the hyperbolic metric, or equivalently, if
\(\widehat{\C}\setminus X\) contains at least three distinct points. In
particular, if \(G\subset \C\) is a plane domain, then \(G\) is hyperbolic if
and only if \(\C\setminus G\) contains at least two points. For such a domain
\(G\), we write \(\lambda_G(z)\,|dz|\) for its hyperbolic metric and
\(\lambda_G\) for the corresponding density.

In \emph{Research Problems in Function Theory}, Hayman and Lingham recorded the following problem of W.\,H.\,J.~Fuchs~\cite[Problem~7.73]{HaymanLingham2019}.

\begin{problem}[Fuchs]\label{prob:fuchs}
Let \(D_1\) and \(D_2\) be domains contained in \(\{z\in\mathbb C:|z|<R\}\), and let
\(\lambda_1(z)\,|dz|\) and \(\lambda_2(z)\,|dz|\) denote their hyperbolic metrics.
What is the smallest constant \(A=A(R)\) such that the hyperbolic metric
\(\lambda(z)\,|dz|\) of \(D_1\cap D_2\) satisfies
\[
\lambda(z) < A\bigl(\lambda_1(z)+\lambda_2(z)\bigr)?
\]
\end{problem}

Hayman and Lingham \cite[pp.~185--186]{HaymanLingham2019} also emphasized both the appeal and the difficulty of this problem:
\begin{quote}
\emph{Personally, I am very fond of Problem 7.73, which is due to Fuchs. \dots\ Once again, this question is easy to state, and easy to understand, but apparently hard to make progress on. It also gives rise to many subsidiary questions: is the lower bound for \(A\) attained? If so, what can be said about the class of domains for which it is attained? Is the question easier if \(D_1\), \(D_2\) are simply connected? Is it the case that \(A\) is independent of \(R\)? Can anything similar be said about the corresponding lower bound on \(\lambda(z)\)?}
\end{quote}

Motivated by these questions, for domains \(D_1,D_2\subset \D_R:=\{z\in\C:|z|<R\}\) and a point \(z\in D_1\cap D_2\), we consider the ratio
\[
Q(D_1,D_2;z)
:=
\frac{\lambda_{D_1\cap D_2}(z)}
{\lambda_{D_1}(z)+\lambda_{D_2}(z)}.
\]
Whenever \(z\in D_1\cap D_2\) and \(D_1\cap D_2\) is not connected, the quantity \(\lambda_{D_1\cap D_2}(z)\) is understood as the hyperbolic density at \(z\) of the connected component of \(D_1\cap D_2\) containing \(z\). In particular, all inequalities involving \(\lambda_{D_1\cap D_2}(z)\) are interpreted componentwise. Moreover, we consider \(\lambda_{D_1\cap D_2}(z)\) only when the connected component containing \(z\) is hyperbolic; this assumption will always be satisfied in the situations treated below. Since every such component is contained in \(\D_R\), it is automatically hyperbolic.

Let \(\mathcal H_R\) denote the class of hyperbolic domains in \(\mathbb C\) contained in \(\D_R\), and let \(\mathcal S_R\subset \mathcal H_R\) denote the subclass of simply connected domains. We define the upper extremal constants
\[
A_{\mathrm{gen}}(R)
:=
\sup\left\{
Q(D_1,D_2;z): D_1,D_2\in\mathcal H_R,\ z\in D_1\cap D_2
\right\},
\]
and
\[
A_{\mathrm{sc}}(R)
:=
\sup\left\{
Q(D_1,D_2;z): D_1,D_2\in\mathcal S_R,\ z\in D_1\cap D_2
\right\}.
\]
Similarly, we define the lower extremal constants
\[
B_{\mathrm{gen}}(R)
:=
\inf\left\{
Q(D_1,D_2;z): D_1,D_2\in\mathcal H_R,\ z\in D_1\cap D_2
\right\},
\]
and
\[
B_{\mathrm{sc}}(R)
:=
\inf\left\{
Q(D_1,D_2;z): D_1,D_2\in\mathcal S_R,\ z\in D_1\cap D_2
\right\}.
\]

The quantities \(A_{\mathrm{gen}}(R)\) and \(A_{\mathrm{sc}}(R)\) are the optimal constants for the non-strict inequality
\[
\lambda_{D_1\cap D_2}(z)
\le
A\bigl(\lambda_{D_1}(z)+\lambda_{D_2}(z)\bigr),
\]
in the general and simply connected settings, respectively. Similarly, \(B_{\mathrm{gen}}(R)\) and \(B_{\mathrm{sc}}(R)\) are the optimal constants for the reverse non-strict inequality
\[
\lambda_{D_1\cap D_2}(z)
\ge
B\bigl(\lambda_{D_1}(z)+\lambda_{D_2}(z)\bigr).
\]

In relation to the original formulation of Hayman and Lingham \cite{HaymanLingham2019}, which uses the strict inequality \(<\), the extremal constants \(A_{\mathrm{gen}}(R)\) and \(A_{\mathrm{sc}}(R)\) should be interpreted as the corresponding threshold values. More precisely, if \(c\) denotes the supremum of the quotient \(Q(D_1,D_2;z)\) over the class under consideration, then every \(A>c\) is admissible for the strict inequality
\[
\lambda_{D_1\cap D_2}(z)
<
A\bigl(\lambda_{D_1}(z)+\lambda_{D_2}(z)\bigr).
\]
Moreover, the endpoint value \(A=c\) is admissible if and only if the supremum \(c\) is not attained. Thus the strict version of the problem is governed by the same extremal value \(c\), while the additional issue is whether the endpoint is admissible, equivalently, whether the supremum is attained.

More precisely, the problems considered here are to determine these extremal constants, to understand whether they depend on \(R\), and to decide whether the corresponding supremum or infimum is attained and, if so, to describe the associated extremal configurations.

We begin with the simplest of these questions, namely the dependence on \(R\). By the conformal invariance of the hyperbolic metric under dilations (Lemma~\ref{lem:scaling}), all four extremal constants are in fact independent of \(R\). Thus, throughout the paper, we may normalize \(R=1\). We shall use the normalization
\[
\lambda_{\D}(z)=\frac{1}{1-|z|^2}, \qquad z\in\D:=\{z\in\C:|z|<1\}.
\]
Any other standard normalization differs from this one by a universal positive constant, and therefore does not affect the extremal problems considered here.

The main purpose of this paper is to address the remaining questions, namely to determine the exact values of these extremal constants, to study the issue of attainability, and, in the case of the infimum problem, to characterize completely when equality occurs.

Our first main result solves Problem~\ref{prob:fuchs}. Specifically, we show that there is no finite universal upper bound: for every \(A>0\), there exist domains \(D_1,D_2\subset \D_R\) and a point \(z\in D_1\cap D_2\) such that
\[
\lambda_{D_1\cap D_2}(z)
>
A\bigl(\lambda_{D_1}(z)+\lambda_{D_2}(z)\bigr).
\]
Equivalently, we prove the following theorem.

\begin{theorem}\label{thm:gen-upper}
For every \(R>0\), one has
\[
A_{\mathrm{gen}}(R)=+\infty.
\]
\end{theorem}

The proof is based on the explicit family
\[
D_1^{(a)}=\D\setminus\{a\},
\qquad
D_2^{(a)}=\D\setminus\{-a\},
\qquad 0<a<1.
\]
For each punctured disk, the hyperbolic density at the origin is of order \((a\log(1/a))^{-1}\), whereas the density on the intersection admits a lower bound given by the density of \(\C\setminus\{\pm a\}\), which is of order \(a^{-1}\). Consequently, the quotient is bounded below by a positive constant multiple of \(\log(1/a)\), and in particular tends to $+\infty$.

Our second main result identifies the sharp upper extremal constant in the simply connected case.

\begin{theorem}\label{thm:sc-upper}
For every \(R>0\), one has
\[
A_{\mathrm{sc}}(R)=1.
\]
\end{theorem}

The proof of Theorem~\ref{thm:sc-upper} is based on an interpretation of the hyperbolic density at the origin in terms of logarithmic capacity for simply connected subdomains of \(\D\). After moving the base point to the origin by a disk automorphism, each simply connected domain \(G\subset \D\) gives rise, via inversion, to a \emph{continuum} \(K_G\), that is, a compact connected set, and \(\lambda_G(0)\) is expressed in terms of \(\capf(K_G)\). Thus, if \(G_1,G_2\subset \D\) are the normalized domains under consideration, then the density on the relevant component of \(G_1\cap G_2\) corresponds to \(\capf(K_{G_1} \cup K_{G_2})\). Murai's subadditivity theorem (Theorem~\ref{thm:murai}) then yields the sharp upper bound.

Finally, we determine the sharp lower extremal constant in both the general and simply connected settings and characterize the equality case for the lower bound.

\begin{theorem}\label{thm:lower}
For every \(R>0\), one has
\[
B_{\mathrm{gen}}(R)=B_{\mathrm{sc}}(R)=\frac12.
\]
Moreover, in both classes \(\mathcal H_R\) and \(\mathcal S_R\), the infimum is
attained, and equality \(Q(D_1,D_2;z)=1/2\) holds if and only if \(D_1=D_2\).
\end{theorem}

\subsection{Related multiplicative intersection inequalities}\label{subsec:previous-work}

A different but closely related inequality was proved by Solynin. For two
hyperbolic domains \(D_1,D_2\) such that \(D_1\cup D_2\) is hyperbolic, one has
\begin{equation}\label{eq:previous_work_v1}
\lambda_{D_1\cap D_2}(z)\,\lambda_{D_1\cup D_2}(z)
\le
\lambda_{D_1}(z)\,\lambda_{D_2}(z),
\qquad z\in D_1\cap D_2.
\end{equation}
In the two-domain form stated above, see also \cite[Theorem~1.1]{KrausRoth2016}
for a direct proof. Solynin established a more general version in the setting
of ordered systems of domains; see \cite[Theorem~2]{Solynin1999}. Since
\eqref{eq:previous_work_v1} is multiplicative rather than additive, it does not
resolve Problem~\ref{prob:fuchs}.

\subsection{Paper organization}\label{subsec:organization}

Section~\ref{sec:preliminaries} collects the background material and auxiliary results on the hyperbolic metric and logarithmic capacity used in the sequel. In Section~\ref{sec:general-upper} we prove Theorem~\ref{thm:gen-upper} by means of punctured-disk examples. Section~\ref{sec:sc-upper} is devoted to the proof of Theorem~\ref{thm:sc-upper}. In Section~\ref{sec:lower-bounds} we prove Theorem~\ref{thm:lower} and characterize the equality case for the lower bound. Finally, Section~\ref{sec:conclusion} discusses the remaining open problem concerning the attainability of \(A_{\mathrm{sc}}(R)=1\).

\section{Preliminaries}\label{sec:preliminaries}

We collect the standard facts on the hyperbolic metric and logarithmic capacity that will be used in the proofs. For general background on the hyperbolic metric, see \cite{Ahlfors1973,BeardonMinda2007}.

\subsection{Basic facts about the hyperbolic metric}\label{subsec:hyperbolic-basics}

A basic characterization that we shall use repeatedly is the following. If \(G\) is a hyperbolic domain and \(\pi:\D\to G\) is a universal covering map, then the hyperbolic metric on \(G\) is the unique conformal metric satisfying
\begin{equation}\label{eq:cover-characterization}
\lambda_G(\pi(\zeta))\,|\pi'(\zeta)|=\lambda_{\D}(\zeta),\qquad \zeta\in\D.
\end{equation}

We begin with the covering-space form of conformal invariance.

\begin{lemma}\label{lem:covering-isometry}
Let \(p:\widetilde G\to G\) be a holomorphic covering map between hyperbolic domains. Then
\[
\lambda_G(p(w))\,|p'(w)|=\lambda_{\widetilde G}(w),\qquad w\in\widetilde G.
\]
\end{lemma}

\begin{proof}
Let \(\pi:\D\to \widetilde G\) be a universal covering map. Then \(p\circ\pi:\D\to G\) is also a universal covering map. Applying \eqref{eq:cover-characterization} first to \(\pi\) and then to \(p\circ\pi\), we obtain
\[
\lambda_{\widetilde G}(\pi(\zeta))\,|\pi'(\zeta)|=\lambda_{\D}(\zeta)
\]
and
\[
\lambda_G\bigl(p(\pi(\zeta))\bigr)\,|(p\circ\pi)'(\zeta)|=\lambda_{\D}(\zeta),
\]
for every \(\zeta\in\D\). Comparing the two identities gives
\[
\lambda_G\bigl(p(\pi(\zeta))\bigr)\,|p'(\pi(\zeta))|=\lambda_{\widetilde G}(\pi(\zeta)).
\]
Since \(\pi\) is surjective, the identity holds for every point of \(\widetilde G\).
\end{proof}

The usual biholomorphic invariance is an immediate consequence.

\begin{lemma}\label{lem:conf}
If \(f:G\to H\) is biholomorphic between hyperbolic domains, then
\[
\lambda_H(f(z))\,|f'(z)|=\lambda_G(z),\qquad z\in G.
\]
\end{lemma}

\begin{proof}
A biholomorphism is, in particular, a holomorphic covering map, so the conclusion follows from Lemma~\ref{lem:covering-isometry}.
\end{proof}

We shall also use the monotonicity of the hyperbolic metric with respect to domain inclusion.

\begin{lemma}\label{lem:mono}
If \(G\subsetneq H\) are hyperbolic domains, then
\[
\lambda_H(z)<\lambda_G(z),\qquad z\in G.
\]
\end{lemma}

\begin{proof}
Apply the Schwarz--Pick lemma to the inclusion map
\[
i\colon G\hookrightarrow H,\qquad i(z)=z;
\]
see, for example, \cite[Theorem~10.5]{BeardonMinda2007}. Since \(i'(z)=1\),
we obtain
\[
\lambda_H(z)=\lambda_H(i(z))\,|i'(z)|\le \lambda_G(z),
\qquad z\in G.
\]
It remains to show that equality cannot occur. Suppose, to the contrary, that
for some \(z_0\in G\) one has
\[
\lambda_H(z_0)=\lambda_G(z_0).
\]
Then equality holds at \(z_0\) in the Schwarz--Pick inequality for \(i\). By
the equality statement in \cite[Theorem~10.5]{BeardonMinda2007}, the map \(i\)
must be a covering of \(H\). However, the inclusion map \(i\colon G\hookrightarrow H\)
is surjective only if \(G=H\), contradicting the assumption \(G\subsetneq H\).
Therefore equality is impossible, and hence
\[
\lambda_H(z)<\lambda_G(z),\qquad z\in G.
\qedhere\]
\end{proof}

The next lemma records the scaling behavior that is responsible for the independence of \(R\).

\begin{lemma}\label{lem:scaling}
Let \(t>0\) and let \(G\) be a hyperbolic domain. Then
\[
\lambda_{tG}(tz)=\frac1t\,\lambda_G(z),\qquad z\in G,
\]
where \(tG:=\{tw:w\in G\}\).
\end{lemma}

\begin{proof}
Apply Lemma~\ref{lem:conf} to the biholomorphism \(f(z)=tz\) from \(G\) onto \(tG\).
\end{proof}

The next two lemmas give the explicit formulas needed for the punctured-disk construction in Section~\ref{sec:general-upper}.

\begin{lemma}\label{lem:punctured-disk}
Let \(\D^*:=\D\setminus\{0\}\). Then
\[
\lambda_{\D^*}(w)=\frac{1}{2|w|\log(1/|w|)},\qquad 0<|w|<1.
\]
\end{lemma}

\begin{proof}
This is standard; see~\cite[Section~12.1]{BeardonMinda2007}.
\end{proof}

As an immediate consequence, we obtain the density at the origin of a one-punctured disk.

\begin{lemma}\label{lem:one-puncture}
Fix \(a\in(0,1)\). Then
\[
\lambda_{\D\setminus\{a\}}(0)=\frac{1-a^2}{2a\log(1/a)}.
\]
The same formula holds for \(\D\setminus\{-a\}\).
\end{lemma}

\begin{proof}
Consider the disk automorphism
\[
\phi_a(z)=\frac{z-a}{1-\overline{a}z},
\]
which maps \(\D\setminus\{a\}\) biholomorphically onto \(\D^*\) and satisfies
\[
\phi_a(0)=-a,
\qquad
\phi_a'(0)=1-a^2.
\]
By Lemma~\ref{lem:conf} and Lemma~\ref{lem:punctured-disk},
\[
\lambda_{\D\setminus\{a\}}(0)=\lambda_{\D^*}(-a)\,|\phi_a'(0)|
=\frac{1-a^2}{2a\log(1/a)}.
\]
The same computation with \(-a\) in place of \(a\) gives the second statement.
\end{proof}

\subsection{Intersection components of simply connected domains}\label{subsec:intersection-topology}

We next record a small topological lemma, which shows that, inside a disk, connected components of intersections of simply connected domains remain simply connected.

\begin{lemma}\label{lem:component-intersection-sc}
Let \(G_1,G_2\subset \DR\) be simply connected domains, and let \(U\) be a connected component of \(G_1\cap G_2\). Then \(U\) is simply connected.
\end{lemma}

\begin{proof}
This is standard. For instance, after translating a point of \(U\) to the origin,
the assertion follows from the reduced-intersection property for simply connected
plane domains; see \cite[Definition~3.5 and property \((\mathrm{RI}1)\)]{BauerKrausRothWegert2010}.
\end{proof}

\subsection{Logarithmic capacity and inversion}\label{subsec:capacity}

We now turn to logarithmic capacity, which provides a convenient reformulation of the hyperbolic density at the origin for simply connected subdomains of~\(\D\). For the conformal-mapping characterization used below, see \cite[Section~2]{LiesenSeteNasser2017} or \cite[Theorem~3.5]{Saff2010}. For general background on logarithmic capacity, Green functions, and the extended complex plane, we also refer the reader to \cite{Ransford1995}.

Let \(K\subset\C\) be a compact set such that the unbounded component
\(\Omega_K\) of \(\widehat{\C}\setminus K\) is simply connected, and assume that
\(K\) is not a singleton. Then there exists a unique conformal map
\[
\Phi_K:\Omega_K\to\{w:|w|>1\}
\]
satisfying
\[
\Phi_K(\infty)=\infty,
\qquad
\Phi_K'(\infty):=\lim_{z\to\infty}\frac{\Phi_K(z)}{z}>0.
\]
Moreover,
\[
\Phi_K(z)=\frac{z}{\mu}+\mu_0+O\!\left(\frac1z\right)
\qquad (z\to\infty)
\]
for some \(\mu>0\). Under these assumptions, the logarithmic capacity of \(K\) is
\[
\capf(K)=\mu=\frac1{\Phi_K'(\infty)}.
\]

Equivalently, if
\[
g_K:=\Phi_K^{-1}:\{|w|>1\}\to\Omega_K,
\]
then
\[
g_K(\infty)=\infty,
\qquad
g_K'(\infty):=\lim_{w\to\infty}\frac{g_K(w)}{w}=\capf(K)>0.
\]
We call \(g_K\) the \emph{normalized exterior conformal map} of \(K\). All sets \(K\) considered below satisfy these assumptions.

We first record the basic transformation rule for capacity.

\begin{lemma}\label{lem:cap-transform}
Let \(K,L\subset\C\) be bounded continua whose unbounded complementary components are simply connected. Assume in addition that neither $K$ nor $L$ is a singleton. Let
\[
h:\Omega_K\to \Omega_L
\]
be conformal with
\[
h(\infty)=\infty,
\qquad
h'(\infty):=\lim_{z\to\infty}\frac{h(z)}{z}>0.
\]
Then
\[
\capf(L)=h'(\infty)\,\capf(K).
\]
\end{lemma}

\begin{proof}
Let \(g_K:\{|w|>1\}\to \Omega_K\) be the normalized exterior conformal map of \(K\). Then \(h\circ g_K\) is the normalized exterior conformal map of \(L\), and
\[
(h\circ g_K)'(\infty)=h'(\infty)g_K'(\infty)=h'(\infty)\,\capf(K).
\]
By definition, \((h\circ g_K)'(\infty)=\capf(L)\).
\end{proof}

To carry out the sharpness argument in Section~\ref{sec:sc-upper}, we will need an explicit conformal description of certain slit exterior domains in terms of the Joukowski map.

\begin{lemma}\label{lem:joukowski-slits}
Let
\[
J(w):=w+\frac1w,
\qquad
\Psi(z):=\frac12\bigl(z+\sqrt{z^2-4}\bigr),
\]
where \(\sqrt{z^2-4}\) denotes the branch on \(\widehat{\C}\setminus[-2,2]\) satisfying
\[
\sqrt{z^2-4}\sim z \qquad (z\to\infty).
\]
Then \(\Psi\) maps \(\widehat{\C}\setminus[-2,2]\) conformally onto \(\{w:|w|>1\}\), and \(J\) is its inverse. In particular,
\[
J(\infty)=\infty,
\qquad
J'(\infty)=1.
\]

Fix \(t>1\). Then \(J\) maps \(\{w:|w|>1\}\setminus[1,t]\) conformally onto \(\widehat{\C}\setminus[-2,t+t^{-1}]\), and maps \(\{w:|w|>1\}\setminus\bigl([-t,-1]\cup[1,t]\bigr)\) conformally onto \(\widehat{\C}\setminus[-(t+t^{-1}),t+t^{-1}]\).
\end{lemma}

\begin{proof}
For \(z\in\widehat{\C}\setminus[-2,2]\), the function \(\Psi(z)\) is holomorphic and satisfies
\[
\Psi(z)^2-z\Psi(z)+1=0,
\]
hence
\[
J(\Psi(z))=\Psi(z)+\frac1{\Psi(z)}=z.
\]
If \(|\Psi(z_0)|=1\) for some \(z_0\in\widehat{\C}\setminus[-2,2]\), then
\[
z_0=J(\Psi(z_0))\in J(\partial\D)=[-2,2],
\]
a contradiction. Since \(\widehat{\C}\setminus[-2,2]\) is connected and \(\Psi(z)\sim z\) as \(z\to\infty\), we conclude that
\[
|\Psi(z)|>1,\qquad z\in\widehat{\C}\setminus[-2,2].
\]

Now let \(|w|>1\). If \(J(w)\in[-2,2]\), then the quadratic equation
\[
u^2-J(w)u+1=0
\]
would have both roots on the unit circle, whereas its roots are exactly \(w\) and \(1/w\). This is impossible. Hence \(J(w)\in\widehat{\C}\setminus[-2,2]\). Since \(\Psi(J(w))\) is a root of the same quadratic and has modulus \(>1\), it must equal \(w\). Therefore \(J\) and \(\Psi\) are inverse conformal maps between \(\widehat{\C}\setminus[-2,2]\) and \(\{w:|w|>1\}\). The statement about \(J'(\infty)\) follows from
\[
J(w)=w+O(1/w)\qquad (w\to\infty).
\]

Now fix \(t>1\) and set \(s=t+t^{-1}\). Since \(x\mapsto x+x^{-1}\) is strictly increasing on \([1,\infty)\),
\[
J([1,t])=[2,s].
\]
By oddness of \(J\),
\[
J([-t,-1])=[-s,-2].
\]
Because \(J\) is injective on \(\{w:|w|>1\}\), we obtain
\[
J(\{w:|w|>1\}\setminus[1,t])
=J(\{w:|w|>1\})\setminus J([1,t])
=(\widehat{\C}\setminus[-2,2])\setminus[2,s]
=\widehat{\C}\setminus[-2,s],
\]
and similarly,
\begin{align*}
J\left(\{w:|w|>1\}\setminus\bigl([-t,-1]\cup[1,t]\bigr)\right)
&=J(\{w:|w|>1\})\setminus\bigl(J([-t,-1])\cup J([1,t])\bigr)
\\ &=(\widehat{\C}\setminus[-2,2])\setminus\bigl([-s,-2]\cup[2,s]\bigr)
=\widehat{\C}\setminus[-s,s].
\end{align*}
This proves the lemma.
\end{proof}

We also need the logarithmic capacity of a real interval, which will be used together with the preceding conformal model.

\begin{lemma}\label{lem:interval-cap}
If \(\alpha<\beta\), then
\[
\capf([\alpha,\beta])=\frac{\beta-\alpha}{4}.
\]
\end{lemma}

\begin{proof}
This is classical; see \cite[Corollary~5.2.4]{Ransford1995}.
\end{proof}

For simply connected subdomains of the unit disk, the hyperbolic density at the origin is exactly a logarithmic capacity.

\begin{lemma}\label{lem:cap-identity}
Let \(G\subset\D\) be a simply connected domain and assume that \(0\in G\). Define
\[
K_G:=\left\{\frac1\zeta:\zeta\in \whC\setminus G\right\}.
\]
Then \(K_G\) is a bounded continuum containing \(\overline{\D}\), and
\[
\lambda_G(0)=\capf(K_G).
\]
\end{lemma}

\begin{proof}
Since \(G\) is simply connected, the set \(\whC\setminus G\) is connected. As \(G\subset\D\), we have
\[
\whC\setminus G\supset \whC\setminus\D.
\]
Applying the inversion \(h(\zeta)=1/\zeta\), we see that \(K_G=h(\whC\setminus G)\) is connected and contains
\[
h(\whC\setminus\D)=\overline{\D}.
\]
Moreover, since \(0\in G\) and \(G\) is open, there exists \(r>0\) such that \(\{|\zeta|<r\}\subset G\). Hence
\[
\whC\setminus G\subset \{\zeta:|\zeta|\ge r\}\cup\{\infty\},
\]
and therefore
\[
K_G\subset \left\{w:|w|\le \frac1r\right\}.
\]
Thus \(K_G\) is bounded. Since it is also closed in \(\C\), it is compact. Hence \(K_G\) is a bounded continuum containing \(\overline{\D}\).

Let \(f:\D\to G\) be the Riemann map normalized by
\[
f(0)=0,
\qquad
f'(0)>0.
\]
By Lemma~\ref{lem:conf} and the normalization \(\lambda_{\D}(0)=1\), we have \(\lambda_G(0)\,f'(0)=1\), so that
\[
\lambda_G(0)=\frac1{f'(0)}.
\]

Now define
\[
g(w):=\frac1{f(1/w)},\qquad |w|>1.
\]
Since \(1/w\in\D\) for \(|w|>1\), the function \(g\) is holomorphic on \(\{w:|w|>1\}\). We claim that \(g\) maps \(\{w:|w|>1\}\) conformally onto the unbounded component \(\Omega_{K_G}\) of \(\whC\setminus K_G\). Indeed, the inversion \(h(\zeta)=1/\zeta\) maps \(G\) onto the component of \(\whC\setminus K_G\) containing \(\infty\). Since \(0\in G\), this component is the unbounded component \(\Omega_{K_G}\). As \(f\) is conformal from \(\D\) onto \(G\), it follows that \(g=h\circ f\circ (1/w)\) is conformal from \(\{w:|w|>1\}\) onto \(\Omega_{K_G}\).

Finally, using the expansion of \(f\) at the origin,
\[
f(\zeta)=f'(0)\zeta+O(\zeta^2)\qquad (\zeta\to0),
\]
we obtain
\[
f(1/w)=\frac{f'(0)}{w}+O\!\left(\frac1{w^2}\right)
\qquad (w\to\infty),
\]
and hence
\[
g(w)=\frac1{f(1/w)}
=\frac{w}{f'(0)}+O(1)
\qquad (w\to\infty).
\]
Therefore
\[
g'(\infty)=\frac1{f'(0)}.
\]

By definition of logarithmic capacity, the logarithmic capacity of \(K_G\) is exactly the coefficient \(g'(\infty)\) of its normalized exterior conformal map. Thus
\[
\capf(K_G)=g'(\infty)=\frac1{f'(0)}=\lambda_G(0).
\]
This completes the proof.
\end{proof}

The simply connected upper bound will rest on the following subadditivity theorem of Murai~\cite[Corollary~21]{Murai1992}.

\begin{theorem}[\cite{Murai1992}]\label{thm:murai}
If \(K_1,K_2\subset\C\) are bounded continua with \(K_1\cap K_2\neq\emptyset\), then
\[
\capf(K_1\cup K_2)\le\capf(K_1)+\capf(K_2).
\]
\end{theorem}

\begin{remark}\label{rem:murai-not-strict}
The inequality in Theorem~\ref{thm:murai} can be an equality. Indeed, by Lemma~\ref{lem:interval-cap},
\[
\capf([0,2])=\frac12=\frac14+\frac14=\capf([0,1])+\capf([1,2]).
\]
\end{remark}

Finally, we identify the hyperbolic density on the relevant component of an intersection with the capacity of a union.

\begin{lemma}\label{lem:intersection-capacity}
Let \(G_1,G_2\subset\D\) be simply connected domains containing \(0\), and let \(U\) be the connected component of \(G_1\cap G_2\) that contains \(0\). For \(j=1,2\), let \(K_j:=K_{G_j}\) be the continuum from Lemma~\ref{lem:cap-identity}. Then
\[
\lambda_U(0)=\capf(K_1\cup K_2).
\]
\end{lemma}

\begin{proof}
By Lemma~\ref{lem:component-intersection-sc}, the domain \(U\) is simply connected. Let \(f:\D\to U\) be the Riemann map normalized by
\[
f(0)=0,
\qquad
f'(0)>0.
\]
Set
\[
g(w):=\frac{1}{f(1/w)},\qquad |w|>1.
\]
As in the proof of Lemma~\ref{lem:cap-identity}, \(g\) is conformal from \(\{|w|>1\}\) onto \(h(U)\), where \(h(\zeta):=1/\zeta\), and
\[
g'(\infty)=\frac{1}{f'(0)}=\lambda_U(0).
\]
For \(j=1,2\), we have \(K_j=h(\whC\setminus G_j)\). Since \(h\) is a bijection of \(\whC\), it preserves complements, and hence
\[
\whC\setminus K_j
=\whC\setminus h(\whC\setminus G_j)
=h(G_j).
\]
Therefore
\[
h(G_1\cap G_2)=h(G_1)\cap h(G_2)
=(\whC\setminus K_1)\cap(\whC\setminus K_2)
=\whC\setminus (K_1\cup K_2).
\]
Since \(U\) is a connected component of \(G_1\cap G_2\) and \(h\) is a homeomorphism of \(\whC\), the set \(h(U)\) is a connected component of \(h(G_1\cap G_2)\). Moreover, \(0\in U\), so \(\infty=h(0)\in h(U)\). Hence \(h(U)\) is the connected component of \(\whC\setminus (K_1\cup K_2)\) containing \(\infty\), that is, the unbounded component \(\Omega_{K_1\cup K_2}\). Thus \(g\) is the normalized exterior conformal map of \(K_1\cup K_2\). Therefore
\[
\capf(K_1\cup K_2)=g'(\infty)=\lambda_U(0).
\qedhere\]
\end{proof}

\section{The general upper bound is infinite}\label{sec:general-upper}

The proof of Theorem~\ref{thm:gen-upper} is based on symmetric punctured disks. Each one-punctured disk has density at the origin of order \((a\log(1/a))^{-1}\), whereas the density on the intersection is bounded below by the density of the twice-punctured plane, which is of order \(a^{-1}\). The extra logarithmic factor forces the quotient to diverge.

\begin{proof}[Proof of Theorem~\ref{thm:gen-upper}]
By Lemma~\ref{lem:scaling}, it suffices to consider the case \(R=1\).

For \(a\in(0,1/2)\), set
\[
D_1^{(a)}:=\D\setminus\{a\},
\qquad
D_2^{(a)}:=\D\setminus\{-a\}.
\]
Then
\[
D_1^{(a)}\cap D_2^{(a)}=\D\setminus\{\pm a\},
\]
which is a hyperbolic domain containing the origin.

By Lemma~\ref{lem:one-puncture},
\[
\lambda_{D_1^{(a)}}(0)=\lambda_{D_2^{(a)}}(0)=\frac{1-a^2}{2a\log(1/a)}.
\]
Hence
\begin{equation}\label{eq:sum-one-puncture}
\lambda_{D_1^{(a)}}(0)+\lambda_{D_2^{(a)}}(0)
=\frac{1-a^2}{a\log(1/a)}.
\end{equation}

To estimate \(\lambda_{\D\setminus\{\pm a\}}(0)\) from below, consider the larger hyperbolic domain $\C\setminus\{\pm a\}$. Since $\D\setminus\{\pm a\}\subset \C\setminus\{\pm a\}$, Lemma~\ref{lem:mono} gives
\begin{equation}\label{eq:monotone-step}
\lambda_{\D\setminus\{\pm a\}}(0)\ge \lambda_{\C\setminus\{\pm a\}}(0).
\end{equation}
Let
\[
m:=\lambda_{\C\setminus\{\pm1\}}(0).
\]
Then \(0<m<\infty\), because \(\C\setminus\{\pm1\}\) is hyperbolic and the covering characterization \eqref{eq:cover-characterization} shows that the hyperbolic density of a hyperbolic domain is finite and strictly positive at every interior point. By Lemma~\ref{lem:scaling},
\begin{equation}\label{eq:scaling-plane}
\lambda_{\C\setminus\{\pm a\}}(0)=\frac{1}{a}\,\lambda_{\C\setminus\{\pm1\}}(0)=\frac{m}{a}.
\end{equation}
Combining \eqref{eq:monotone-step} and \eqref{eq:scaling-plane}, we obtain
\begin{equation}\label{eq:lower-bound-main}
\lambda_{\D\setminus\{\pm a\}}(0)\ge \frac{m}{a}.
\end{equation}

Finally, from \eqref{eq:sum-one-puncture} and \eqref{eq:lower-bound-main},
\[
\frac{\lambda_{D_1^{(a)}\cap D_2^{(a)}}(0)}{\lambda_{D_1^{(a)}}(0)+\lambda_{D_2^{(a)}}(0)}
=
\frac{\lambda_{\D\setminus\{\pm a\}}(0)}{\lambda_{D_1^{(a)}}(0)+\lambda_{D_2^{(a)}}(0)}
\ge
\frac{m/a}{(1-a^2)/(a\log(1/a))}
=
\frac{m\log(1/a)}{1-a^2}.
\]
As \(a\to0^+\), the right-hand side tends to \(+\infty\). Hence \(A_{\mathrm{gen}}(1)=+\infty\), and the general case follows again from Lemma~\ref{lem:scaling}.
\end{proof}

\section{The sharp upper bound in the simply connected case}\label{sec:sc-upper}

We first prove the upper estimate \(A_{\mathrm{sc}}(R)\le 1\), and then show that the constant \(1\) is sharp.

\begin{proof}[Proof of the upper bound in Theorem~\ref{thm:sc-upper}]
Fix \(R>0\), simply connected domains \(D_1,D_2\subset \D_R\), and a point \(z\in D_1\cap D_2\). Set
\[
E_j:=\frac1R D_j\subset\D,
\qquad
\zeta:=\frac{z}{R}\in E_1\cap E_2.
\]
Let
\[
\phi_\zeta(w):=\frac{w-\zeta}{1-\overline\zeta w},
\qquad w\in\D,
\]
so that \(\phi_\zeta\) is an automorphism of \(\D\) with \(\phi_\zeta(\zeta)=0\). Define
\[
\widetilde{G}_j:=\phi_\zeta(E_j),\qquad j=1,2.
\]
Then \(\widetilde{G}_1,\widetilde{G}_2\subset\D\) are simply connected and contain the origin.

Let \(U\) be the connected component of \(\widetilde{G}_1\cap \widetilde{G}_2\) that contains \(0\). For each \(j\in\{1,2\}\), let \(K_j:=K_{\widetilde{G}_j}\) be the continuum from Lemma~\ref{lem:cap-identity}. Since each \(K_j\) contains \(\overline\D\), we have \(K_1\cap K_2\neq\emptyset\). Therefore Theorem~\ref{thm:murai} gives
\[
\capf(K_1\cup K_2)\le\capf(K_1)+\capf(K_2).
\]
By Lemma~\ref{lem:intersection-capacity} and Lemma~\ref{lem:cap-identity}, this becomes
\[
\lambda_U(0)\le \lambda_{\widetilde{G}_1}(0)+\lambda_{\widetilde{G}_2}(0).
\]
Since \(U\) is the component of \(\widetilde{G}_1\cap \widetilde{G}_2\) containing \(0\), our convention implies
\[
\lambda_{\widetilde{G}_1\cap \widetilde{G}_2}(0)=\lambda_U(0).
\]
Using Lemma~\ref{lem:conf},
\[
\lambda_{E_j}(\zeta)=\lambda_{\widetilde{G}_j}(0)\,|\phi_\zeta'(\zeta)|,
\qquad
\lambda_{E_1\cap E_2}(\zeta)=\lambda_{\widetilde{G}_1\cap \widetilde{G}_2}(0)\,|\phi_\zeta'(\zeta)|.
\]
Hence
\[
\lambda_{E_1\cap E_2}(\zeta)\le\lambda_{E_1}(\zeta)+\lambda_{E_2}(\zeta).
\]
Finally, Lemma~\ref{lem:scaling} yields
\[
\lambda_{E_j}(\zeta)=R\lambda_{D_j}(z),
\qquad
\lambda_{E_1\cap E_2}(\zeta)=R\lambda_{D_1\cap D_2}(z).
\]
Dividing by \(R\), we obtain
\[
\lambda_{D_1\cap D_2}(z)\le\lambda_{D_1}(z)+\lambda_{D_2}(z).
\]
This proves \(A_{\mathrm{sc}}(R)\le1\).
\end{proof}

\begin{proof}[Proof of the sharpness in Theorem~\ref{thm:sc-upper}]
By the upper bound proved above, whenever \(D_1,D_2\in\mathcal S_R\), one has \(Q(D_1,D_2;z)\le 1\). Therefore, it remains to construct a family for which \[Q(D_1,D_2;z)\to 1.\]

By Lemma~\ref{lem:scaling}, the ratio is invariant under dilations, so it is enough to work with \(R=1\), that is, inside the unit disk \(\D\). Fix \(a\in(0,1)\) and define
\[
D_3^{(a)}:=\D\setminus [a,1),
\qquad
D_4^{(a)}:=\D\setminus (-1,-a].
\]
Both domains are simply connected and contain the origin.

We first compute \(\lambda_{D_3^{(a)}}(0)\). By Lemma~\ref{lem:cap-identity},
\[
\lambda_{D_3^{(a)}}(0)=\capf(K_3^{(a)}),
\]
where
\[
K_3^{(a)}
=
\left\{\frac1\zeta:\zeta\in \widehat{\C}\setminus D_3^{(a)}\right\}
=
\overline{\D}\cup [1,a^{-1}].
\]
Hence the unbounded component of \(\widehat{\C}\setminus K_3^{(a)}\) is
\[
\widetilde{\Omega}_3:=\{w:|w|>1\}\setminus [1,a^{-1}].
\]
By Lemma~\ref{lem:joukowski-slits}, the Joukowski map \(J(w)=w+1/w\) maps \(\widetilde{\Omega}_3\) conformally onto \(\widehat{\C}\setminus[-2,a+a^{-1}]\), and \(J'(\infty)=1\). Therefore Lemma~\ref{lem:cap-transform} and Lemma~\ref{lem:interval-cap} give
\[
\lambda_{D_3^{(a)}}(0)
=
\capf(K_3^{(a)})
=
\capf([-2,a+a^{-1}])
=
\frac{a+a^{-1}+2}{4}
=
\frac{(1+a)^2}{4a}.
\]
By symmetry,
\[
\lambda_{D_4^{(a)}}(0)=\frac{(1+a)^2}{4a}.
\]

Now consider the intersection:
\[
D_3^{(a)}\cap D_4^{(a)}
=
\D\setminus\bigl([a,1)\cup(-1,-a]\bigr).
\]
This domain is connected. By Lemma~\ref{lem:intersection-capacity},
\[
\lambda_{D_3^{(a)}\cap D_4^{(a)}}(0)
=
\capf(K_{34}^{(a)}),
\]
where
\[
K_{34}^{(a)}
=
\overline{\D}\cup[-a^{-1},-1]\cup[1,a^{-1}].
\]
Thus the unbounded component of \(\widehat{\C}\setminus K_{34}^{(a)}\) is
\[
\widetilde{\Omega}_{34}:=\{w:|w|>1\}\setminus\bigl([-a^{-1},-1]\cup[1,a^{-1}]\bigr).
\]
Again by Lemma~\ref{lem:joukowski-slits}, the map \(J(w)=w+1/w\) sends \(\widetilde{\Omega}_{34}\) conformally onto \(\widehat{\C}\setminus[-(a+a^{-1}),a+a^{-1}]\), with \(J'(\infty)=1\). Hence
\[
\lambda_{D_3^{(a)}\cap D_4^{(a)}}(0)
=
\capf(K_{34}^{(a)})
=
\capf([-(a+a^{-1}),a+a^{-1}])
=
\frac{a+a^{-1}}{2}
=
\frac{1+a^2}{2a}.
\]
Therefore
\[
Q(D_3^{(a)},D_4^{(a)};0)=\frac{\lambda_{D_3^{(a)}\cap D_4^{(a)}}(0)}
{\lambda_{D_3^{(a)}}(0)+\lambda_{D_4^{(a)}}(0)}
=
\frac{(1+a^2)/(2a)}{(1+a)^2/(2a)}
=
\frac{1+a^2}{(1+a)^2}.
\]
Letting \(a\to0^+\) yields \((1+a^2)/(1+a)^2 \longrightarrow 1\). Hence \(A_{\mathrm{sc}}(1)\ge 1\). Combined with the upper bound, this yields
\[
A_{\mathrm{sc}}(1)=1.
\]
By scaling invariance, it follows that \(A_{\mathrm{sc}}(R)=1\) for every \(R>0\).
\end{proof}

\section{Sharp lower bounds and equality cases}\label{sec:lower-bounds}

We now turn to the lower extremal constants and the characterization of equality in the lower-bound inequality.

\begin{proof}[Proof of Theorem~\ref{thm:lower}]
Let \(D_1,D_2\in\mathcal H_R\), let \(z\in D_1\cap D_2\), and let \(U\) be the
connected component of \(D_1\cap D_2\) containing \(z\). By definition,
\[
\lambda_{D_1\cap D_2}(z)=\lambda_U(z).
\]
Since \(U\subset D_1\) and \(U\subset D_2\), Lemma~\ref{lem:mono} yields
\[
\lambda_{D_1\cap D_2}(z)=\lambda_U(z)\ge \lambda_{D_1}(z),
\qquad
\lambda_{D_1\cap D_2}(z)=\lambda_U(z)\ge \lambda_{D_2}(z).
\]
Hence
\[
\lambda_{D_1\cap D_2}(z)
\ge \max\{\lambda_{D_1}(z),\lambda_{D_2}(z)\}
\ge \frac12\bigl(\lambda_{D_1}(z)+\lambda_{D_2}(z)\bigr).
\]
Therefore
\[
B_{\mathrm{gen}}(R)\ge \frac12.
\]
Since \(\mathcal S_R\subset \mathcal H_R\), we also have
\[
B_{\mathrm{gen}}(R)\le B_{\mathrm{sc}}(R).
\]

On the other hand, if \(D_1=D_2\), then
\[
\lambda_{D_1\cap D_2}(z)=\lambda_{D_1}(z)=\lambda_{D_2}(z)
\]
for every \(z\in D_1\), and hence
\[
Q(D_1,D_2;z)=\frac12.
\]
It follows that
\[
B_{\mathrm{sc}}(R)\le \frac12.
\]
Combining the above inequalities, we obtain
\[
\frac12\le B_{\mathrm{gen}}(R)\le B_{\mathrm{sc}}(R)\le \frac12.
\]
Therefore
\[
B_{\mathrm{gen}}(R)=B_{\mathrm{sc}}(R)=\frac12.
\]

It remains to characterize the equality case. Suppose that
\[
Q(D_1,D_2;z)=\frac12.
\]
Then equality must hold throughout the chain
\[
\lambda_{D_1\cap D_2}(z)=\lambda_U(z)
\ge \max\{\lambda_{D_1}(z),\lambda_{D_2}(z)\}
\ge \frac12\bigl(\lambda_{D_1}(z)+\lambda_{D_2}(z)\bigr).
\]
Hence
\[
\lambda_U(z)=\lambda_{D_1}(z)=\lambda_{D_2}(z).
\]
Since \(U\subset D_1\) and \(U\subset D_2\), Lemma~\ref{lem:mono} implies that
\(U=D_1\) and \(U=D_2\). Thus \(D_1=D_2\).

Conversely, if \(D_1=D_2\), then, as observed above,
\[
Q(D_1,D_2;z)=\frac12
\]
for every \(z\in D_1\). This completes the proof.
\end{proof}

\section{Concluding remarks and open problems}\label{sec:conclusion}

In this paper we have resolved, with one exception, the questions raised by Fuchs and discussed by Hayman and Lingham in \cite{HaymanLingham2019} concerning the four extremal constants
\[
A_{\mathrm{gen}}(R),\quad A_{\mathrm{sc}}(R),\quad B_{\mathrm{gen}}(R),\quad B_{\mathrm{sc}}(R).
\]
The only remaining open problem concerns the endpoint case in the simply connected upper estimate. Namely, although \(A_{\mathrm{sc}}(R)=1\), it is unknown whether this endpoint value is attained. Equivalently, it is unknown whether equality can occur in
\[
\lambda_{D_1\cap D_2}(z)\le \lambda_{D_1}(z)+\lambda_{D_2}(z)
\]
for some simply connected domains \(D_1,D_2\subset \D_R\) and some point \(z\in D_1\cap D_2\).

In fact, we are led to conjecture a stronger four-domain additive inequality, reminiscent of inequality~\eqref{eq:previous_work_v1}.

\begin{conjecture}\label{conj:four-domain-additive}
Let \(D_1,D_2\subsetneq \C\) be simply connected domains, let
\(z\in D_1\cap D_2\), and assume that \(D_1\cup D_2\) is hyperbolic. Then
\[
\lambda_{D_1\cap D_2}(z)+\lambda_{D_1\cup D_2}(z)
\le
\lambda_{D_1}(z)+\lambda_{D_2}(z).
\]
\end{conjecture}
If Conjecture~\ref{conj:four-domain-additive} is true, then the supremum \(A_{\mathrm{sc}}(R)=1\)
in Theorem~\ref{thm:sc-upper} cannot be attained. Indeed, suppose that for some simply connected
domains \(D_1,D_2\subset \D_R\) and some \(z\in D_1\cap D_2\) one had
\[
\lambda_{D_1\cap D_2}(z)=\lambda_{D_1}(z)+\lambda_{D_2}(z).
\]
Then Conjecture~\ref{conj:four-domain-additive} would imply
\[
\lambda_{D_1\cup D_2}(z)
\le
\lambda_{D_1}(z)+\lambda_{D_2}(z)-\lambda_{D_1\cap D_2}(z)
=0,
\]
which is impossible, since the hyperbolic density is strictly positive on every hyperbolic domain.
Thus no extremal configuration can realize the supremum \(A_{\mathrm{sc}}(R)=1\).

In this sense, Conjecture~\ref{conj:four-domain-additive} would explain why the constant \(1\) is
sharp in the simply connected upper-bound problem, yet unattainable.

\begin{remark}\label{rem:four-domain-additive-equality}
Equality can occur in Conjecture~\ref{conj:four-domain-additive}. Indeed, for \(a\in(0,1)\), let
\[
D_3^{(a)}:=\D\setminus [a,1),
\qquad
D_4^{(a)}:=\D\setminus (-1,-a].
\]
Then \(D_3^{(a)}\) and \(D_4^{(a)}\) are simply connected, \(0\in D_3^{(a)}\cap D_4^{(a)}\), and \(D_3^{(a)}\cup D_4^{(a)}=\D\). Moreover, by the proof of Theorem~\ref{thm:sc-upper}, we know that
\[
\lambda_{D_3^{(a)}}(0)=\lambda_{D_4^{(a)}}(0)=\frac{(1+a)^2}{4a},
\qquad
\lambda_{D_3^{(a)}\cap D_4^{(a)}}(0)=\frac{1+a^2}{2a},
\qquad
\lambda_{D_3^{(a)}\cup D_4^{(a)}}(0)=\lambda_{\D}(0)=1.
\]
Hence
\[
\lambda_{D_3^{(a)}\cap D_4^{(a)}}(0)+\lambda_{D_3^{(a)}\cup D_4^{(a)}}(0)
=
\lambda_{D_3^{(a)}}(0)+\lambda_{D_4^{(a)}}(0).
\]
Thus equality is possible in Conjecture~\ref{conj:four-domain-additive}.
\end{remark}

\end{document}